\documentclass [12pt] {article}
\usepackage{amssymb}
\usepackage{amsmath}
\usepackage{amsfonts}

\newtheorem{theorem}{Theorem}

\newtheorem{corollary}[theorem]{Corollary}

\newtheorem{definition}[theorem]{Definition}

\newtheorem{lemma}[theorem]{Lemma}

\newtheorem{problem}[theorem]{Problem}
\newtheorem{proposition}[theorem]{Proposition}

\newenvironment{proof}[1][Proof]{\noindent\textbf{#1.} }{\ \rule{0.5em}{0.5em}}
\begin{document}

\title{Euclidean arrangements in Banach spaces}
\author{Daniel J. Fresen\thanks{Yale University, Department of Mathematics, daniel.fresen@yale.edu}}
\maketitle

\begin{abstract}
We study the way in which the Euclidean subspaces of a Banach space fit
together, somewhat in the spirit of the Ka\v{s}in decomposition. The main
tool that we introduce is an estimate regarding the convex hull of a convex
body in John's position with a Euclidean ball of a given radius, which leads
to a new and simplified proof of the randomized isomorphic Dvoretzky
theorem. Our results also include a characterization of spaces with
nontrivial cotype in terms of arrangements of Euclidean subspaces.

\bigskip

\textbf{Note:} This is not necessarily the final version of the paper. Please consult the published version when it becomes available.
\end{abstract}

\section{Introduction}

A fundamental result in the geometry of Banach spaces is Dvoretzky's theorem
(see e.g. \cite{MS,Sch2}), which states that any Banach space $X$\ of
dimension $n\in \mathbb{N}$ is richly endowed with approximately Euclidean
subspaces of dimension $\left\lfloor c\log n\right\rfloor $. Besides knowing
that there are many Euclidean subspaces, it is not known precisely how these
subspaces are arranged within $X$. In the case $X=\ell _{1}^{n}$ it has been
shown, going back to the work of Ka\v{s}in \cite{Ka}, that there exist
mutually orthogonal subspaces $E_{1},E_{2}\subset \mathbb{R}^{n}$ such that $%
\mathbb{R}^{n}=E_{1}\oplus E_{2}$ and for all $i\in \{1,2\}$ and all $x\in
E_{i}$,%
\begin{equation*}
c_{1}|x|\leq \frac{1}{\sqrt{n}}||x||_{1}\leq c_{2}|x|
\end{equation*}%
where $c_{1},c_{2}>0$ are universal constants, $|x|=\left(
\sum_{i=1}^{n}x_{i}{}^{2}\right) ^{1/2}$ and $||x||_{1}=%
\sum_{i=1}^{n}|x_{i}| $. The same decomposition was shown to hold for spaces
with universally bounded volume ratio \cite{Szarek, Szarek TJ} (see Section %
\ref{Background a7} for more details). Using the results just mentioned, it
is easy to show that there exists an orthonormal basis for $\mathbb{R}^{n}$,
say $(e_{i})_{1}^{n}$, such that for all $(0.99n)$-sparse vectors $a\in 
\mathbb{R}^{n}$,%
\begin{equation}
c_{1}\left( \sum_{i=1}^{n}a_{i}^{2}\right) ^{1/2}\leq \frac{1}{\sqrt{n}}%
\left\Vert \sum_{i=1}^{n}a_{i}e_{i}\right\Vert _{1}\leq c_{2}\left(
\sum_{i=1}^{n}a_{i}^{2}\right) ^{1/2}  \label{acting Hilbert}
\end{equation}%
In this paper we seek a collection of Euclidean subspaces of a general
finite dimensional Banach space that fit together like a grid in such a way
so that, with respect to a particular basis, vectors with various regularity
properties act as though they were in a Hilbert space, in the sense of (\ref%
{acting Hilbert}). We are interested in both isomorphic and almost-isometric
type estimates. There are two main types of regularity that we impose. The
first is sparsity, and the second is simplicity of support, measured in
terms of cyclic length and Kolmogorov complexity.

The most interesting example is $\ell _{\infty }^{n}$ which has a $%
(1+\varepsilon )$ Ka\v{s}in-style decomposition but not the richer
arrangement of $(1+\varepsilon )$ Euclidean subspaces as described above.
Numerous questions remain unsolved, even for $\ell _{\infty }^{n}$.

\section{\label{Main results sec2}Main results}

\subsection{The isomorphic theory}

A key innovation of the paper is the following lemma which uses a result of
Vershynin \cite{Versh}\ on contact points of $\partial B_{X}$ with the John
ellipsoid, see Theorem \ref{res Ver}\ and Lemma \ref{contact pts}.

\begin{lemma}
\label{unconditional iso symm}There exist universal constants $c,c^{\prime
},c_{1},c_{2}>0$ with $c^{\prime }>2c_{2}^{2}$ such that the following
holds. Consider any Banach space $X$ of dimension $n\in \mathbb{N}$ and let $%
t\in \mathbb{R}$ with $c_{1}\leq t\leq c_{2}\sqrt{n}$. Identify $X$ with $%
\mathbb{R}^{n}$ so that $B_{X}$ is in John's position. Let $%
K_{t}=conv\{tB_{2}^{n},B_{X}\}$. Let $M_{t}$ and $b_{t}$ denote the median
and maximum of the Minkowski functional of $K_{t}$ on $S^{n-1}$. Then%
\begin{equation*}
\frac{M_{t}}{b_{t}}\geq ct\sqrt{\frac{1}{n}\log \left( \frac{c^{\prime }n}{%
t^{2}}\right) }
\end{equation*}
\end{lemma}

The body $K_{(\rho )}=B_{X}\cap \rho B_{2}^{n}$, related to $K_{t}$ by
duality, has appeared before in the literature. For this we refer the reader
to \cite{GMT, LTJ} and the references therein. An immediate consequence is a
version of the randomized isomorphic Dvoretzky theorem of Litvak, Mankiewicz
and Tomczak-Jaegermann \cite{LTM} (see also the original papers by Milman
and Schechtman \cite{MiSc1,MiSc2}).

\begin{corollary}
\label{randomized isomorphic Dvoretzky}There exist universal constants $%
c,c_{1},c_{2},C>0$ such that the following is true. Let $(X,\left\Vert \cdot
\right\Vert )$ be a real Banach space of dimension $n\in \mathbb{N}$ that we
identify with $\mathbb{R}^{n}$ so that the ellipsoid of maximum volume in $%
B_{X}=\left\{ x:\left\Vert x\right\Vert \leq 1\right\} $ is the standard
Euclidean ball $B_{2}^{n}$. Let $1\leq k\leq n$ and let $E\in G_{n,k}$ be a
random subspaces uniformly distributed in $G_{n,k}$. Then with probability
at least $1-C\exp \left( -c\max \left\{ k,\log n\right\} \right) $ the
following event occurs. For all $x\in E$, 
\begin{equation*}
c_{1}M^{(k)}\left\vert x\right\vert \leq \left\Vert x\right\Vert \leq c_{2}%
\sqrt{\frac{k+\log n}{\log \left( 1+n/k)\right) }}M^{(k)}\left\vert
x\right\vert
\end{equation*}%
where $M^{(k)}$ is the average value of the Minkowski functional of $%
conv\left( B_{X}\cup tB_{2}^{n}\right) $ on $S^{n-1}$, with $t=c_{2}\sqrt{%
(k+\log n)/\log \left( 1+n/k)\right) }$.
\end{corollary}

\begin{proof}
Assume first that $k<c^{\prime }n$ for a sufficiently small $c^{\prime }>0$.
Set $K_{t}=conv\left( B_{X}\cup tB_{2}^{n}\right) $ and apply Lemma \ref%
{unconditional iso symm} followed by Milman's general Dvoretzky theorem (see
eg. Theorem \ref{general Dvoret}) with $\varepsilon =1/2$. We conclude that
with probability at least $1-C\exp \left( -c\max \left\{ k,\log n\right\}
\right) $ the following event occurs: for all $x\in E$, $c_{1}M^{(k)}\left%
\vert x\right\vert \leq \left\Vert x\right\Vert _{K_{t}}\leq
c_{2}M^{(k)}\left\vert x\right\vert $. Then note that $\left\Vert
x\right\Vert _{K_{t}}\leq \left\Vert x\right\Vert \leq t\left\Vert
x\right\Vert _{K_{t}}$. If $c^{\prime }n\leq k\leq n$ the result follows by
John's theorem.
\end{proof}

For a vector $a\in \mathbb{R}^{n}$, let $a^{\flat }\in \mathbb{R}^{n}$
denote the indicator function of the support of $a$, i.e.%
\begin{equation*}
a_{i}^{\flat }=\left\{ 
\begin{array}{ccc}
0 & : & a_{i}=0 \\ 
1 & : & a_{i}\neq 0%
\end{array}%
\right.
\end{equation*}%
Let $\left\Vert a\right\Vert _{0}=\left\Vert a^{\flat }\right\Vert
_{1}=\left\vert \left\{ i:a_{i}\neq 0\right\} \right\vert $ denote the
sparsity of $a$, let $\left\Vert a\right\Vert _{Kol}=C_{Kol}\left( a^{\flat
}\right) $ denote the Kolmogorov complexity of $a^{\flat }$, and let 
\begin{equation*}
\left\Vert a\right\Vert _{cyc}=\min \left\{ k\leq n:\exists m\leq n,(k\leq
i<n)\Rightarrow a_{(m+i)\text{mod}n}=0\right\}
\end{equation*}%
denote the cyclic length of the support of $a$.

Another direct consequence of Lemma \ref{unconditional iso symm} (or just as
well Corollary \ref{randomized isomorphic Dvoretzky}) is as follows.

\begin{corollary}
\label{gen Dvoretzky type decomp}There exist universal constants $%
c,c^{\prime },C>0$ such that the following is true. Let $(X,\left\Vert \cdot
\right\Vert )$ be a real Banach space of dimension $n\in \mathbb{N}$ that we
identify with $\mathbb{R}^{n}$ so that the ellipsoid of maximum volume in $%
B_{X}=\left\{ x:\left\Vert x\right\Vert \leq 1\right\} $ is the standard
Euclidean ball $B_{2}^{n}$. Let $(e_{i})_{1}^{n}$ be a random orthonormal
basis for $\mathbb{R}^{n}$ generated by the action of a random orthogonal
matrix uniformly distributed in $O(n)$. Then with probability at least $%
1-Cn^{-c}$, the following event occurs. For all $a\in \mathbb{R}^{n}$, 
\begin{equation*}
cM_{D(a)}\left( \sum_{i=1}^{n}a_{i}^{2}\right) ^{1/2}\leq \left\Vert
\sum_{i=1}^{n}a_{i}e_{i}\right\Vert \leq D(a)M_{D(a)}\left(
\sum_{i=1}^{n}a_{i}^{2}\right) ^{1/2}
\end{equation*}%
where $M_{D(a)}$ is the average value of the Minkowski functional of $%
conv\left( B_{X}\cup D(a)B_{2}^{n}\right) $ on $S^{n-1}$, and the distortion 
$D(a)$ can be written as%
\begin{equation}
D(a)=c^{\prime }\min \left\{ \left\Vert a\right\Vert _{0}^{1/2},\left( \frac{%
\left\Vert a\right\Vert _{cyc}+\log n}{\log \left( 1+n\left\Vert
a\right\Vert _{cyc}^{-1}\right) }\right) ^{1/2},\left( \frac{\left\Vert
a\right\Vert _{0}+\left\Vert a\right\Vert _{Kol}+\log n}{\log \left(
1+n\left( \left\Vert a\right\Vert _{0}+\left\Vert a\right\Vert _{Kol}\right)
^{-1}\right) }\right) ^{1/2}\right\}  \label{distotzion}
\end{equation}%
Note also that $M_{t}$ is non-increasing in $t$ while $tM_{t}$ is
non-decreasing.
\end{corollary}

The same dependence on $\left\Vert a\right\Vert _{0}$ can be achieved with
the use of the Dvoretzky-Rogers factorization \cite{BoSz, Gian, Szarek2,
SzTa, Versh}, but only on a subspace of proportional dimension. Note that in
Corollary \ref{gen Dvoretzky type decomp}, logarithmic sparsity alone is not
enough to guarantee bounded Euclidean distortion.

\begin{problem}
\label{Euclidean grid}Does there exist a universal constant $C>0$ and a
sequence $(\omega _{n})_{1}^{\infty }$ with $\lim_{n\rightarrow \infty
}\omega _{n}=\infty $ such that the following is true? For every $n\in 
\mathbb{N}$ and any real Banach space $(X,\left\Vert \cdot \right\Vert )$ of
dimension $n$, there is a basis $(e_{i})_{1}^{n}$ for $X$ such that for any $%
a\in \mathbb{R}^{n}$ with $\left\Vert a\right\Vert _{0}\leq \omega _{n}$,%
\begin{equation*}
\left( \sum_{i=1}^{n}a_{i}^{2}\right) ^{1/2}\leq \left\Vert
\sum_{i=1}^{n}a_{i}e_{i}\right\Vert \leq C\left(
\sum_{i=1}^{n}a_{i}^{2}\right) ^{1/2}
\end{equation*}%
Can one take $\omega _{n}=c\log n$? Can one take $(e_{i})_{1}^{n}$ to be
orthonormal with respect to the John ellipsoid of $B_{X}$?
\end{problem}

In Section \ref{section almost iso thry} we show that the corresponding $%
(1+\varepsilon )$ estimate does not hold in $\ell _{\infty }^{n}$.

Lemma \ref{unconditional iso symm} can also be used to prove an inequality
for the distribution of norms on $S^{n-1}$ that is very similar to a result
of Schechtman and Schmuckenschl\"{a}ger \cite{ScSc}. This in turn shows that
bodies in John's position with low Dvoretzky dimension have a large
Klartag-Vershynin parameter $d_{u}(K)$, a parameter which is significant
partly because of its relation to the outer inclusion in Dvoretzky's theorem 
\cite{KlVe}. These results are presented in Corollaries \ref{concentration
bound} and \ref{parameter est}.

Finally, let us note that a forthcomming paper of Chasapis and Giannopoulos 
\cite{ChGi} further explores consequences of Lemma \ref{unconditional iso
symm}, which includes an isomorphic version of the global Dvoretzky theorem
of Bourgain, Lindenstrauss and Milman.

\subsection{\label{section almost iso thry}The almost-isometric theory}

Let $(X,\left\Vert \cdot \right\Vert )$ be a real Banach space of dimension $%
n\in \mathbb{N}$. Using Corollary \ref{gen Dvoretzky type decomp} followed
by a further application of the randomized Dvoretzky theorem, it follows
that we may identify $X$ with $\mathbb{R}^{n}$ in such a manner so that $%
c_{1}B_{2}^{n}\subset \mathcal{E}\subset c_{2}B_{2}^{n}$, where $\mathcal{E}$
is the John ellipsoid of $B_{X}$, and we may write $X$ as the internal
direct sum of subspaces%
\begin{equation*}
X=\oplus _{i=1}^{N}H_{i}
\end{equation*}%
where $\dim (H_{i})\geq c(\varepsilon )\log n$ and the subspaces $%
(H_{i})_{1}^{N}$ are pairwise orthogonal, and for each $1\leq i\leq N$ and
each $x\in H_{i}$,%
\begin{equation*}
(1-\varepsilon )\widetilde{M}|x|\leq \left\Vert x\right\Vert \leq
(1+\varepsilon )\widetilde{M}|x|
\end{equation*}%
where $\widetilde{M}$ depends on $X$. However for many spaces such as those
with a symmetric basis and those that come with a pre-packaged coordinate
system, such as $\ell _{\infty }^{n}$, the Euclidean structure associated to
the (exact) John ellipsoid is of particular importance. Answering a question
that we posed in an earlier draft of this paper, Konstantin Tikhomirov gave
a proof of the following result which we discuss further in Section \ref{ai
dec}. We thank him for allowing us to include it here.

\begin{theorem}
\label{new}There exists a universal constant $c>0$ with the following
property. Let $(X,\left\Vert \cdot \right\Vert )$ be a real Banach space of
dimension $n\in \mathbb{N}$ that we identify with $\mathbb{R}^{n}$ so that
the ellipsoid of maximum volume in $B_{X}=\left\{ x:\left\Vert x\right\Vert
\leq 1\right\} $ is the standard Euclidean ball $B_{2}^{n}$ and let $c(\log
\log n)^{3/2}/(\log n)^{1/2}<\varepsilon <1/2$. Then there exists a
decomposition $X=\oplus _{i=1}^{N}H_{i}$ into mutually orthogonal subspaces $%
(H_{i})_{1}^{N}$ of $\dim (H_{i})\geq c\varepsilon ^{2}(\log \varepsilon
^{-1})^{-1}\log n$ such that for all $1\leq i\leq N$ and all $x\in H_{i}$,%
\begin{equation*}
(1-\varepsilon )M^{\sharp }|x|\leq \left\Vert x\right\Vert \leq
(1+\varepsilon )M^{\sharp }|x|
\end{equation*}%
where $M^{\sharp }$ is the median of $\left\Vert \cdot \right\Vert $ on $%
S^{n-1}$.
\end{theorem}

In Corollary \ref{infinite decomp} we consider the infinite dimensional
case. Theorem \ref{new} guarantees the existence of at least one
decomposition, and we do not know whether it holds for a typical
decomposition. Referring back to Theorem \ref{gen Dvoretzky type decomp}, we
also do not know whether one can achieve a $1+\varepsilon $ estimate for all 
$a$ such that $\left\Vert a\right\Vert _{cyc}\leq c(\varepsilon )\omega _{n}$
for some fixed function $\omega _{n}$ with $\omega _{n}\rightarrow \infty $,
such as $\omega _{n}=\log n$.

\begin{definition}
\label{def loc hilb}Let $X$ be an infinite dimensional Banach space over $%
\mathbb{R}$. We shall say that $X$ satisfies Definition \ref{def loc hilb}
if for all $\varepsilon >0$ and all $k\in \mathbb{N}$, there exists $N\in 
\mathbb{N}$ with the following property. For any finite dimensional subspace 
$E\subset X$ with $\dim (E)>N$, there exists a basis $(e_{i})_{1}^{n}$ for $%
E $ such that for any $a\in \mathbb{R}^{n}$ with $\left\Vert a\right\Vert
_{0}\leq k$,%
\begin{equation}
(1-\varepsilon )\left( \sum_{i=1}^{n}|a_{i}|^{2}\right) ^{1/2}\leq
\left\Vert \sum_{i=1}^{n}a_{i}e_{i}\right\Vert _{X}\leq (1+\varepsilon
)\left( \sum_{i=1}^{n}|a_{i}|^{2}\right) ^{1/2}  \label{loc hilbert def}
\end{equation}
\end{definition}

The space $c_{0}$ does not have this property (see Lemma \ref{c0 not loc
Hilbert}). Going back to the work of Kwapie\'{n} \cite{Kw} and Figiel,
Lindenstrauss and Milman \cite{FLM}, it is well known that there are
intimate connections between the Euclidean structures within a Banach space
and the notions of type and cotype. The $\ell _{\infty }^{n}$ spaces have,
in a sense, the smallest possible collection of Euclidean subspaces. By the
Maurey-Pisier-Krivine theorem, these spaces are excluded as subspaces of $X$
precisely when $X$ has nontrivial cotype. This leads naturally into the
following result.

\begin{theorem}
\label{local Hilbert cotype}An infinite dimensional real Banach space $X$
satisfies Definition \ref{def loc hilb} if and only if it has nontrivial
cotype.
\end{theorem}

Our proof shows that when $X$ has cotype $q<\infty $ and corresponding
cotype constant $\beta \in (0,1]$, such a basis exists provided $%
0<\varepsilon <0.99$ and $k\leq c\beta ^{2}\varepsilon ^{2}(\log
(en/k))^{-1}n^{2/q}$. For cotype 2 spaces the bound for $k$ can be polished
slightly and written as $k\leq c\beta ^{2}\varepsilon ^{2}(\log \beta
^{-1}+\log \varepsilon ^{-1})^{-1}n$.

\subsection{Related observations}

Our results are closely related to the restricted isometry property of Cand%
\`{e}s and Tao \cite{CanTa} which plays a fundamental role in compressed
sensing. They can be understood as generalized restricted
isometry/isomorphism properties for random operators from $\ell _{2}^{n}$
into more general normed spaces. It follows from Theorem \ref{local Hilbert
cotype} that nontrivial cotype is a natural condition to assume in the
following two results.

\begin{proposition}
\label{RIP for cotype}Let $\left( X,\left\Vert \cdot \right\Vert _{X}\right) 
$ and $\left( Y,\left\Vert \cdot \right\Vert _{Y}\right) $ be real Banach
spaces of dimensions $n,m\in \mathbb{N}$ respectively, with cotype $%
q_{1},q_{2}<\infty $ and corresponding cotype constants $\beta _{1},\beta
_{2}\in (0,1]$. Let $0<\varepsilon <0.99$ and consider any $k\in \mathbb{N}$
such that%
\begin{equation*}
k\leq c\varepsilon ^{2}(\log (en/k))^{-1}\min \left\{ \beta
_{1}^{2}n^{2/q_{1}},\beta _{2}^{2}m^{2/q_{2}}\right\} 
\end{equation*}%
Then there exist bases $(e_{i})_{1}^{n}$ and $(f_{i})_{1}^{m}$ in $X$ and $Y$
respectively with the following property. Let $G$ be an $m\times n$ random
matrix with i.i.d. $N(0,1)$ entries. With probability at least%
\begin{equation*}
1-c_{1}\exp \left( -C\beta _{2}^{2}m^{2/q_{2}}\varepsilon ^{2}\right) 
\end{equation*}%
the following event occurs. For all $k$-sparse vectors $a\in \mathbb{R}^{n}$,%
\begin{equation*}
(1-\varepsilon )\left\Vert \sum_{j=1}^{n}a_{j}e_{j}\right\Vert _{X}\leq
\left\Vert \frac{1}{\sqrt{m}}\sum_{i=1}^{m}\sum_{j=1}^{n}G_{ij}a_{j}f_{i}%
\right\Vert _{Y}\leq (1+\varepsilon )\left\Vert
\sum_{j=1}^{n}a_{j}e_{j}\right\Vert _{X}
\end{equation*}%
Here, $c,c_{1},C>0$ are universal constants.
\end{proposition}

Let us say that a vector $x\in X$ is $k$-sparse with respect to a basis $%
(e_{i})_{1}^{n}$ if it can be expressed as a linear combination of no more
than $k$ basis vectors. The Johnson-Lindenstrauss lemma \cite{JohnLind} does
not hold in a general Banach space, even spaces with nontrivial cotype (see 
\cite{JoNa} and the references therein). For vectors that are sparse with
respect to a particular basis (that can be chosen randomly), the situation
is different.

\begin{proposition}
\label{general JL}Let $(X,\left\Vert \cdot \right\Vert )$ be a real Banach
space of dimension $n\in \mathbb{N}$, with cotype $q\in \lbrack 2,\infty )$
and corresponding cotype constant $\beta \in (0,1]$. Let $0<\varepsilon <0.99
$ and let $k\in \mathbb{N}$ such that $k\leq c\beta ^{2}\varepsilon
^{2}(\log (en/k))^{-1}n^{2/q}$. Then there exists a basis $(e_{i})_{1}^{n}$
for $X$ with the following property. Let $\Omega \subset X$ be a finite
collection of vectors that are each $k$-sparse with respect to the given
basis and let $m=\left\lfloor C\varepsilon ^{-2}\log |\Omega |\right\rfloor $%
. Then there exists a linear operator $T:X\rightarrow \ell _{2}^{m}$ such
that for all $x,y\in \Omega $,%
\begin{equation*}
(1-\varepsilon )\left\Vert x-y\right\Vert \leq |Tx-Ty|\leq (1+\varepsilon
)\left\Vert x-y\right\Vert 
\end{equation*}
\end{proposition}

If in the above proposition one insists on a map $Q:X\rightarrow E$, where $%
E $ is a subspace of $X$, then we may use Dvoretzky's theorem and modify the
bound on $m$.

\section{\label{Background a7}Background}

Most of the background material relevant to the paper can be found in \cite%
{AlKa, FLM, Mat, Ma, MS, MiShar, Pisier 0, Pisier}. The letters $c$, $c_{1}$%
, $c_{2}$, $c^{\prime }$, $C$ etc. denote universal constants that take on
different values from one line to the next. They are not arbitrary, but have
very specific numerical values that we do not always have control over. The
symbols $\mathbb{P}$ and $\mathbb{E}$ denote probability and expected value.
For $p\in \lbrack 1,\infty )$, let $||x||_{p}=\left(
\sum_{i=1}^{n}|x_{i}|^{p}\right) ^{1/p}$ denote the $\ell _{p}^{n}$ norm of
a vector $x\in \mathbb{R}^{n}$ and let $|\cdot |=||\cdot ||_{2}$ be the
standard Euclidean norm. For $p=\infty $, $||x||_{\infty }=\max
\{|x_{i}|:1\leq i\leq n\}$. Let $B_{p}^{n}=\{x\in \mathbb{R}%
^{n}:||x||_{p}\leq 1\}$ and $S^{n-1}=\{x\in \mathbb{R}^{n}:|x|=1\}$. The
Grassmannian manifold $G_{n,k}$ consists of all $k$-dimensional linear
subspaces of $\mathbb{R}^{n}$, while the orthogonal group $O(n)$ is the
space of all orthogonal $n\times n$ matrices. The spaces $S^{n-1}$, $G_{n,k}$
and $O(n)$ each have a unique rotational invariant probability measure
called Haar measure, that will be denoted in each case as $\sigma _{n}$. Of
fundamental importance is L\'{e}vy's concentration inequality.

\begin{theorem}
Let $f:S^{n-1}\rightarrow \mathbb{R}$ be a Lipschitz function and let $%
M_{f}=\int_{S^{n-1}}f\,d\sigma _{n}$. Then for all $t>0$,%
\begin{equation}
\sigma _{n}\left\{ \theta \in S^{n-1}:\left\vert f(\theta )-M_{f}\right\vert
<t\cdot Lip(f)\right\} \geq 1-c_{1}e^{-c_{2}nt^{2}}  \label{Levy's ineq}
\end{equation}%
where $c_{1},c_{2}>0$ are universal constants. The same result holds with
the mean $M_{f}$ replaced with anything between (say) the $10^{th}$ and $%
90^{th}$ percentile of $f$, such as the median, and $Lip(f)$ is measured
with respect to either the Euclidean metric on $S^{n-1}$ or the geodesic
distance.
\end{theorem}

A (centrally) symmetric convex body $K\subset \mathbb{R}^{n}$ is a compact,
convex set with non empty interior such that $x\in K$ if and only if $-x\in
K $. The associated Minkowski and dual Minkowski functionals are the norms
defined by%
\begin{eqnarray*}
||x||_{K} &=&\inf \{t\geq 0:x\in tK\} \\
||y||_{K^{\circ }} &=&\max \left\{ \left\langle x,y\right\rangle :x\in
K\right\}
\end{eqnarray*}%
where $\left\langle \cdot ,\cdot \right\rangle $ is the standard Euclidean
inner product. The body%
\begin{equation*}
K^{\circ }=\left\{ y:||y||_{K^{\circ }}\leq 1\right\}
\end{equation*}%
is known as the polar of $K$.

The John ellipsoid of a convex body $K$, denoted $\mathcal{E}_{K}$, is the
ellipsoid of maximal volume contained within $K$. It can be shown via a
compactness argument that such an ellipsoid exists. It is also known that $%
\mathcal{E}_{K}$ is unique. When $\mathcal{E}_{K}=B_{2}^{n}$, we say that $K$
is in John's position, and in this case (assuming $K$ is symmetric)%
\begin{equation*}
B_{2}^{n}\subseteq K\subseteq \sqrt{n}B_{2}^{n}
\end{equation*}%
Two parameters of particular importance are the mean and the maximum,%
\begin{equation}
M(K)=\int_{S^{n-1}}||\theta ||_{K}d\sigma _{n}(\theta )  \label{mean norm}
\end{equation}%
\begin{equation}
b(K)=\max \left\{ ||\theta ||_{K}:\theta \in S^{n-1}\right\}
\label{max norm}
\end{equation}%
When $K$ is in John's position $b\leq 1$, and it can be shown using the
Dvoretzky-Rogers lemma that%
\begin{equation}
M\geq c\sqrt{\frac{\log n}{n}}  \label{mean in John's position}
\end{equation}

Let $(\Omega ,\rho )$ be a compact metric space and $0<\varepsilon <1$. An $%
\varepsilon $-net $\mathcal{N}\subset \Omega $ is a set such that for all $%
\theta \in \Omega $ there exists $\omega \in \mathcal{N}$ such that $\rho
(\theta ,\omega )<\varepsilon $ and for all $\omega _{1},\omega _{2}\in 
\mathcal{N}$, $\rho (\omega _{1},\omega _{2})\geq \varepsilon $. Sometimes
the latter condition is dropped. Such a set can easily be shown to exist. In
the case $\Omega =S^{n-1}$, a volumetric argument yields%
\begin{equation}
|\mathcal{N}|\leq \left( \frac{3}{\varepsilon }\right) ^{n}
\label{card of net}
\end{equation}%
By homogeneity, any $x\in \mathbb{R}^{n}$ can be expressed as $x=|x|\omega
_{0}+x^{\prime }$, where $\omega _{0}\in \mathcal{N}$ and $|x^{\prime
}|<\varepsilon |x|$. Iterating this expression yields%
\begin{equation}
x/|x|=\omega _{0}+\sum_{i=1}^{\infty }\varepsilon _{i}\omega _{i}
\label{net series}
\end{equation}%
where $(\omega _{i})_{0}^{\infty }$ is a sequence in $\mathcal{N}$ and $%
0\leq \varepsilon _{i}<\varepsilon ^{i}$. Applying the triangle inequality
then leads to the following lemma.

\begin{lemma}
\label{net to sphere}Let $||\cdot ||$ be a norm on $\mathbb{R}^{n}$ and $%
\delta \in (0,1/4)$. Let $M>0$ and let $\mathcal{N}$ be a $\delta $-net in $%
S^{n-1}$. Suppose that for all $\omega \in \mathcal{N}$, $(1-\delta )M\leq
||\omega ||\leq (1+\delta )M$. Then for all $x\in \mathbb{R}^{n}$, $%
(1-4\delta )M\leq ||x||\leq (1+4\delta )M$.
\end{lemma}

\begin{theorem}[The general Dvoretzky theorem]
\label{general Dvoret}Let $||\cdot ||$ be a norm on $\mathbb{R}^{n}$ with
parameters $M,b$ as defined by (\ref{mean norm}) and (\ref{max norm})
respectively. Let $0<\varepsilon <0.99$ and $k\leq c_{1}\varepsilon
^{2}M^{2}b^{-2}n$, and let $E\in G_{n,k}$ be any fixed subspace. Let $T$ be
a random orthogonal matrix uniformly distributed in $O(n)$ and let $E=TF$.
Then with probability at least $1-c_{1}\exp \left( -c_{2}\varepsilon
^{2}M^{2}b^{-2}n\right) $ we have that for all $x\in E$, $(1-\varepsilon
)M|x|\leq \left\Vert x\right\Vert \leq (1+\varepsilon )M|x|$.
\end{theorem}

\begin{proof}[Sketch]
Let us start by giving Milman's original proof under the slightly stronger
assumption that $k\leq c_{1}\varepsilon ^{2}(\log \varepsilon
^{-1})^{-1}M^{2}b^{-2}n$. Let $\mathcal{N}$ be an $\varepsilon /4$-net in $%
S(F)=\{x\in F:|x|=1\}$. The epsilon net bound (\ref{card of net}) yields $|%
\mathcal{N}|\leq (12/\varepsilon )^{k}$. It follows from the triangle
inequality and the definition of $b$ that $||\cdot ||$ is Lipschitz on $%
\mathbb{R}^{n}$ with $Lip(||\cdot ||)=b$. For each $\omega \in \mathcal{N}$, 
$T\omega $ is uniformly distributed in $S^{n-1}$. The result then follows
from L\'{e}vy's inequality (\ref{Levy's ineq}) with $t=4^{-1}\varepsilon
Mb^{-1}$, the union bound, and Lemma \ref{net to sphere}. To eliminate the
factor $(\log \varepsilon ^{-1})^{-1}$, one can use a random $n\times n$
matrix $Q$ with i.i.d. standard Gaussian entries, instead of $T\in O(n)$, as
well as a more intricate epsilon net argument. The use of a Gaussian matrix
allows one to take advantage of the fact that if $u\perp v$ then $Qu$ and $%
Qv $ are independent. Eliminating the factor $(\log \varepsilon ^{-1})^{-1}$
using Gaussian processes/matrices was done by Gordon \cite{Gordon} and Schechtman 
\cite{Sch0}. Lastly, the matrix $n^{-1/2}Q$ acts as an approximate isometry
on $F$ (with respect to the Euclidean norm on $F$ and $QF$), and so the
result for $Q$ can then be transferred to the result for $T$.
\end{proof}

If $(F_{i})_{1}^{N}$ is any collection of subspaces of $\mathbb{R}^{n}$,
then we may apply Theorem \ref{general Dvoret} simultaneously to all $F_{i}$
and modify the corresponding probability (using the union bound). This is
usually how we shall apply Theorem \ref{general Dvoret}.

If $X$ is any finite dimensional Banach space, then there exists a basis for 
$X$ such that with respect to this basis, the unit ball $B_{X}=\left\{
x:\left\Vert x\right\Vert \leq 1\right\} $ is in John's position and by (\ref%
{mean in John's position}), we can take $k=\left\lfloor c\varepsilon
^{2}\log n\right\rfloor $. The best known dependence on $\varepsilon $ for
the existence of at least one deterministic subspace $E$ is $c(\varepsilon
)=c\varepsilon (\log \varepsilon ^{-1})^{-2}$ by Schechtman \cite{Sch,Sch2}.

The volume ratio of a convex body $K\subset \mathbb{R}^{n}$ is defined as%
\begin{equation*}
\mathrm{vr}(K)=\left( \frac{\mathrm{vol}_{n}(K)}{\mathrm{vol}_{n}(\mathcal{E}%
_{K})}\right) ^{1/n}
\end{equation*}%
This goes back to the paper of Szarek \cite{Szarek} where it was used
implicitly to prove Ka\v{s}in's result, and then explicitly defined in \cite%
{Szarek TJ} following a suggestion by Pe\l czy\'{n}ski. The volume ratio
theorem states that if $K$ is symmetric, $1\leq k\leq n$, and $E\in G_{n,k}$
is a random subspace of dimension $k$ (uniformly distributed with respect to
the Haar measure on $G_{n,k}$ corresponding to $\mathcal{E}_{K}$), then with
probability at least $1-2^{-n}$,%
\begin{equation*}
(\mathcal{E}_{K}\cap E)\subseteq (K\cap E)\subseteq \left( 4\pi \mathrm{vr}%
(K)\right) ^{n/(n-k)}(\mathcal{E}_{K}\cap E)
\end{equation*}%
For spaces with universally bounded volume ratio, such as $\ell _{1}^{n}$
where $\mathrm{vr}(B_{1}^{n})\leq \sqrt{2\pi /e}$, this gives a version of
Dvoretzky's theorem with proportional dimension, as well as the Ka\v{s}in
decomposition mentioned in the introduction. The probability bound $1-2^{-n}$
is somewhat arbitrary and can be replaced with $1-t^{n}$ for any $t>0$ if we
replace the $4\pi $ with $ct^{-1}$. Setting $t=1/10$ (say), using the union
bound and the inequality $\binom{n}{k}\leq (en/k)^{k}$, this implies
inequality (\ref{acting Hilbert}) for $k=\left\lceil 0.99n\right\rceil $.

A Banach space $E$ embeds (linearly) into a space $Y$ with distortion $%
\gamma \geq 1$, denoted $E\hookrightarrow _{\gamma }Y$, if there exists a
linear subspace $F\subseteq Y$ and a linear bijection $T:E\rightarrow F$
such that $\left\Vert T\right\Vert \cdot \left\Vert T^{-1}\right\Vert
=\gamma $. The Euclidean distortion of $E$ is defined as%
\begin{equation*}
d_{2}(E)=\inf \left\{ \gamma \geq 1:E\hookrightarrow _{\gamma }H\right\}
\end{equation*}%
where $H$ is a suitably large Hilbert space. A space $X$ is finitely
representable in $Y$ if for all finite dimensional subspaces $E\subset X$
and all $\varepsilon >0$, $E\hookrightarrow _{1+\varepsilon }Y$. By
Dvoretzky's theorem, every Hilbert space is finitely representable in every
infinite dimensional Banach space.

The notions of type and cotype capture the spirit of the $L_{p}$ spaces in
an abstract setting. Let $(\varepsilon _{i})_{1}^{\infty }$ denote an i.i.d.
sequence of Rademacher random variables with $\mathbb{P}\left\{ \varepsilon
_{i}=1\right\} =\mathbb{P}\left\{ \varepsilon _{i}=-1\right\} =1/2$. A
Banach space $X$ is said to have type $p\in \lbrack 1,2]$ if there exists $%
\alpha \in \lbrack 1,\infty )$ such that for all finite sequences $%
(x_{i})_{1}^{m}$ in $X$,%
\begin{equation*}
\mathbb{E}\left\Vert \sum_{i=1}^{m}\varepsilon _{i}x_{i}\right\Vert \leq
\alpha \left( \sum_{i=1}^{m}\left\Vert x_{i}\right\Vert ^{p}\right) ^{1/p}
\end{equation*}%
Similarly, $X$ is said to have cotype $q\in \lbrack 2,\infty ]$ if there
exists $\beta \in (0,1]$ such that for all $(x_{i})_{1}^{m}$ in $X$,%
\begin{equation*}
\beta \left( \sum_{i=1}^{n}\left\Vert x_{i}\right\Vert ^{q}\right)
^{1/q}\leq \mathbb{E}\left\Vert \sum_{i=1}^{n}\varepsilon
_{i}x_{i}\right\Vert
\end{equation*}%
with the appropriate interpretation when $q=\infty $. Any Banach space $X$
has type $1$ and cotype $\infty $, and these are referred to as trivial
type/cotype. If $X$ has type $p$ and cotype $q$, then it has type $p^{\prime
}$ and cotype $q^{\prime }$ for all $p^{\prime }\in \lbrack 1,p]$ and $%
q^{\prime }\in \lbrack q,\infty ]$. Type and cotype are inherited by
subspaces, and the space $L_{p}$ $(1\leq p<\infty )$ has type $\min \{p,2\}$
and cotype $\max \{p,2\}$. If $E$ is a finite dimensional space with cotype $%
q<\infty $ and corresponding constant $\beta $, then with respect to the
John ellipsoid of $B_{E}$ inequality (\ref{mean in John's position}) can be
improved to 
\begin{equation}
M\geq c\beta n^{\frac{1}{q}-\frac{1}{2}}  \label{cotype bound}
\end{equation}%
where $n=\dim (E)$, and the general Dvoretzky theorem guarantees the
existence of Euclidean subspaces of dimension $c\varepsilon ^{2}\beta
^{2}n^{2/q}$. Here we use the notation in \cite{FLM} where $\beta \in (0,1]$%
. Some authors refer to $\beta ^{-1}$ as the cotype constant, and others use
yet another definition which is equivalent up to a constant $C_{p}$.

Let $p_{X}$ and $q_{X}$ be the supremum (resp. infimum) over all values of $%
p $ and $q$ such that $X$ has type $p$ and cotype $q$. One of the most
significant results in the theory of type and cotype is the
Maurey-Pisier-Krivine theorem \cite{Kriv, MP}, which builds on work by
Brunel and Sucheston \cite{BrSu}.

\begin{theorem}
If $X$ is infinite dimensional, then $\ell _{p_{X}}$ and $\ell _{q_{X}}$ are
finitely representable in $X$.
\end{theorem}

The following powerful result of Vershynin plays a key role in our work (see
Corollary 5.5 and the discussion on p.269 in \cite{Versh}):

\begin{theorem}
\label{res Ver}There exists a function $\xi :(0,1)\rightarrow (1,\infty )$
such that the following is true. Consider any $\varepsilon \in (0,1)$ and
let $X$ be any $n$-dimensional real Banach space that we identify with $%
\mathbb{R}^{n}$ so that $B_{X}=\left\{ x:\left\Vert x\right\Vert \leq
1\right\} $ is in John's position. Then there exists $m>(1-\varepsilon )n$
and a sequence $(v_{i})_{1}^{m}\subset S^{n-1}\cap \partial B_{X}$ such that
for all $a\in \mathbb{R}^{m}$,%
\begin{equation*}
\xi (\varepsilon )^{-1}\left( \sum_{i=1}^{m}a_{i}^{2}\right) ^{1/2}\leq
\left\vert \sum_{i=1}^{m}a_{i}v_{i}\right\vert \leq \xi (\varepsilon )\left(
\sum_{i=1}^{m}a_{i}^{2}\right) ^{1/2}
\end{equation*}
\end{theorem}

\begin{corollary}
\label{contact pts}Let $(X,\left\Vert \cdot \right\Vert _{X})$ be a real
Banach space of dimension $n\in \mathbb{N}$ and let $m=\left\lceil
n/2\right\rceil $. Then there exists an inner product $\left\langle \cdot
,\cdot \right\rangle _{\sharp }$ on $X$ and a sequence $(u_{i})_{1}^{m}$
that is orthonormal with respect to $\left\langle \cdot ,\cdot \right\rangle
_{\sharp }$ such that for all $1\leq i\leq m$,%
\begin{equation*}
C^{-1}\leq \left\Vert u_{i}\right\Vert _{X}\leq C
\end{equation*}%
\begin{equation*}
C^{-1}\leq \left\Vert u_{i}\right\Vert _{X^{\ast }}\leq C
\end{equation*}%
where $\left\Vert \cdot \right\Vert _{X^{\ast }}$ is the dual norm on $X$
under the duality corresponding to $\left\langle \cdot ,\cdot \right\rangle
_{\sharp }$, i.e. $\left\Vert x\right\Vert _{X^{\ast }}=\sup \left\{
\left\langle x,y\right\rangle _{\sharp }:\left\Vert y\right\Vert _{X}\leq
1\right\} $. Furthermore, the ellipsoid of maximum volume in $B_{X}$,
denoted $\mathcal{E}$, satisfies $c_{1}\mathcal{E}^{\sharp }\subseteq 
\mathcal{E}\subseteq c_{2}\mathcal{E}^{\sharp }$, where $\mathcal{E}^{\sharp
}=\{x\in X:\left\langle x,x\right\rangle _{\sharp }\leq 1\}$.
\end{corollary}

\begin{proof}
Here we review the theory surrounding Vershynin's result. Identify $X$ with $%
\mathbb{R}^{n}$ so that $K=B_{X}$ is in John's position. By Theorem \ref{res
Ver}, there exists a sequence $(v_{i})_{1}^{m}$ of contact points between $%
B_{2}^{n}$ and $\partial K$ such that for any sequence of coefficients $%
(a_{i})_{1}^{m}\in \mathbb{R}^{m}$,%
\begin{equation*}
c_{1}\left( \sum_{i=1}^{m}a_{i}^{2}\right) ^{1/2}\leq \left\vert
\sum_{i=1}^{m}a_{i}v_{i}\right\vert \leq c_{2}\left(
\sum_{i=1}^{m}a_{i}^{2}\right) ^{1/2}
\end{equation*}%
By construction, $\left\Vert v_{i}\right\Vert _{K}=|v_{i}|=1$ for all $1\leq
i\leq m$. These vectors are linearly independent, and can be extended to a
basis for $\mathbb{R}^{n}$, $(v_{i})_{1}^{n}$, such that $(v_{i})_{m+1}^{n}$
are orthonormal and $span\{v_{i}\}_{1}^{m}$ is orthogonal to $%
span\{v_{i}\}_{m+1}^{n}$. Using these properties, for any sequence of
coefficients $(a_{i})_{1}^{n}\in \mathbb{R}^{n}$,%
\begin{equation*}
c_{3}\left( \sum_{i=1}^{n}a_{i}^{2}\right) ^{1/2}\leq \left\vert
\sum_{i=1}^{n}a_{i}v_{i}\right\vert \leq c_{4}\left(
\sum_{i=1}^{n}a_{i}^{2}\right) ^{1/2}
\end{equation*}%
which implies $c\leq ||A||_{2\rightarrow 2}\leq c^{\prime }$ and $c\leq
||A^{-1}||_{2\rightarrow 2}\leq c^{\prime }$, where $A$ is the $n\times n$
matrix with the vectors $(v_{i})_{1}^{n}$ as columns and $\left\Vert \cdot
\right\Vert _{2\rightarrow 2}$ denotes the operator norm of a matrix from $%
\ell _{2}^{n}$ to $\ell _{2}^{n}$. Let $(u_{i})_{1}^{n}$ denote the dual
basis of $(v_{i})_{1}^{n}$, i.e. the columns of $(A^{T})^{-1}$. Since $%
||A^{T}||_{2\rightarrow 2}=||A||_{2\rightarrow 2}$ and $||(A^{T})^{-1}||_{2%
\rightarrow 2}=||A^{-1}||_{2\rightarrow 2}$, we have%
\begin{equation*}
c_{5}\left( \sum_{i=1}^{n}a_{i}^{2}\right) ^{1/2}\leq \left\vert
\sum_{i=1}^{n}a_{i}u_{i}\right\vert \leq c_{6}\left(
\sum_{i=1}^{n}a_{i}^{2}\right) ^{1/2}
\end{equation*}%
By the Hahn-Banach theorem, there exists $(w_{i})_{1}^{m}$ such that $%
\left\Vert w_{i}\right\Vert _{K^{\circ }}=1$ and $\left\langle
v_{i},w_{i}\right\rangle =1$. Since $K^{\circ }\subseteq B_{2}^{n}$, $%
|w_{i}|\leq 1$. Since $\left\vert v_{i}\right\vert ,|w_{i}|\leq 1$ and $%
\left\langle v_{i},w_{i}\right\rangle =1$, it follows that $v_{i}=w_{i}$.
Therefore $\left\Vert v_{i}\right\Vert _{K^{\circ }}=1$. By definition of $%
(u_{i})_{1}^{n}$, $\left\langle u_{i},v_{j}\right\rangle =\delta _{i,j}$,
and therefore for each $1\leq i\leq m$,%
\begin{equation*}
\left\vert a_{i}\right\vert =\left\vert \left\langle
\sum_{j=1}^{m}a_{j}u_{j},v_{i}\right\rangle \right\vert \leq \left\Vert
\sum_{j=1}^{m}a_{j}u_{j}\right\Vert _{K}
\end{equation*}%
which implies that for all $a\in \mathbb{R}^{m}$,%
\begin{equation*}
\max_{1\leq i\leq m}\left\vert a_{i}\right\vert \leq \left\Vert
\sum_{j=1}^{m}a_{j}u_{j}\right\Vert \leq c_{6}\left(
\sum_{i=1}^{m}a_{i}^{2}\right) ^{1/2}
\end{equation*}%
Here we have also used the fact that $K$ is in John's position, which
implies that $\left\Vert \cdot \right\Vert _{K}\leq \left\vert \cdot
\right\vert $. Define $\left\langle x,y\right\rangle _{\sharp }=\left\langle
A^{T}x,A^{T}y\right\rangle $. Since $A^{T}u_{i}=e_{i}$ for each $1\leq i\leq
m$, it follows that $(u_{i})_{1}^{m}$ are orthonormal with respect to $%
\left\langle \cdot ,\cdot \right\rangle _{\sharp }$. The fact that $c_{1}%
\mathcal{E}^{\sharp }\subseteq B_{2}^{n}\subseteq c_{2}\mathcal{E}^{\sharp }$
follows from the fact that $\max \left\{ ||A^{T}||_{2\rightarrow
2},||(A^{T})^{-1}||_{2\rightarrow 2}\right\} \leq c$.
\end{proof}

The Kolmogorov complexity of a finite binary string $b\in \{0,1\}^{\ast }=$ $%
\cup _{n=0}^{\infty }\{0,1\}^{n}$ is defined as (see for example \cite{LiVi})%
\begin{equation*}
C_{Kol}(b)=\min \left\{ \ell (p):p\in \{0,1\}^{\ast },\phi (p)=b\right\}
\end{equation*}%
where $\phi $ is the universal partial recursive function generated by a
specific universal Turing machine, and $\ell (p)$ denotes the length of $p$.
This measures the amount of information contained in the string, which may
be much less than its length due to redundancy and the existence of
patterns. For example a string of the form%
\begin{equation}
b=(0,0,0,\ldots ,1,1,1,\ldots ,0,0,0,\ldots )  \label{binary string}
\end{equation}%
has Kolmogorov complexity at most $c\log (n+2)$, where $n=\ell (b)$. In
order to describe the string, we need to express the fact that a zero never
occurs between two 1's, which can be communicated using at most $c$ bits. We
must then describe the starting point of the $1^{\prime }s$ and the ending
point, which requires at most $2\log _{2}(n+1)+c$ bits. The types of strings
most relevant to us are those with low complexity, such as those of the form
(\ref{binary string}). This differs from the more common situation where one
is interested in strings of high complexity.

\section{Isomorphic Theory}

\begin{lemma}
\label{max order stats}Let $(X_{i})_{1}^{m}$ be an i.i.d. $N(0,1)$ sequence
and $(\log m)^{-1/2}\leq s\leq 1-c(\log m)^{-1}$. With probability at least $%
0.52$, 
\begin{equation*}
\left\vert \left\{ i:s\sqrt{\log m}\leq X_{i}\leq 3\sqrt{\log m}\right\}
\right\vert \geq c^{\prime }m^{1-s^{2}}
\end{equation*}
\end{lemma}

\begin{proof}
Let $\Phi $ denote the cumulative standard normal distribution function. For
all $t\geq 1$ (see e.g. \cite{Dud}),%
\begin{equation*}
\frac{\phi (t)}{2t}\leq 1-\Phi (t)\leq \frac{\phi (t)}{t}
\end{equation*}%
where $\phi $ is the standard normal density. This implies that%
\begin{equation*}
\mathbb{P}\left\{ s\sqrt{\log m}\leq X_{i}\leq 3\sqrt{\log m}\right\} \geq
cm^{-s^{2}}
\end{equation*}%
Let $Y=\left\vert \left\{ i:s\sqrt{\log m}\leq X_{i}\leq 3\sqrt{\log m}%
\right\} \right\vert $. By taking the constant $c$ in the statement of the
lemma to be sufficiently large, it follows that $\mathbb{E}Y\geq 1000$ and%
\begin{equation*}
\sqrt{\mathrm{Var}(Y)}\leq \left( \mathbb{E}Y\right) ^{1/2}\leq \frac{%
\mathbb{E}Y}{30}
\end{equation*}%
It then follows from Chebyshev's inequality that $\mathbb{P}\left\{ Y\geq 
\mathbb{E}Y/10\right\} \geq 0.52$.
\end{proof}

\begin{proof}[Proof of Lemma \protect\ref{unconditional iso symm}]
Without loss of generality we may assume that $n>n_{0}$, where $n_{0}\in 
\mathbb{N}$ is a sufficiently large universal constant. Let $A=\left\{ x\in
X:\left\Vert x\right\Vert \leq 1\right\} $ and $m=\left\lceil
n/2\right\rceil $. Identify $X$ with $\mathbb{R}^{n}$ so that the inner
product $\left\langle \cdot ,\cdot \right\rangle _{\sharp }$ from Corollary %
\ref{contact pts} is the standard Euclidean inner product, and consider the
vectors $(u_{i})_{1}^{m}$ as in Corollary \ref{contact pts}. With this
coordinate structure, $A$ is not necessarily in John's position. Near the
end of the proof we shall use a second coordinate structure. After applying
an orthogonal transformation, we may assume that $u_{i}=e_{i}$ for all $%
1\leq i\leq m$, where $(e_{i})_{1}^{n}$ are the standard basis vectors of $%
\mathbb{R}^{n}$. Let $A_{t}=conv\{A,tB_{2}^{n}\}$. Then $||y||_{A_{t}^{\circ
}}=\max \{||y||_{A^{\circ }},t|y|\}$. Let $\theta \in S^{n-1}$ be a random
point uniformly distributed on the sphere. We can simulate $\theta =X/|X|$,
where $(X_{i})_{1}^{n}$ are i.i.d. $N(0,1)$ variables. Set $s=(1-2\log
(ct)/\log m)^{1/2}$ in which case $t=cm^{(1-s^{2})/2}$. By Lemma \ref{max
order stats}, with probability at least $0.52$, $\left\vert \Omega
\right\vert \geq cm^{1-s^{2}}=ct^{2}$, where%
\begin{equation*}
\Omega =\left\{ 1\leq i\leq m:s\sqrt{\log m}\leq X_{i}\leq 3\sqrt{\log m}%
\right\}
\end{equation*}%
Let $z=\sum_{i\in \Omega }e_{i}$. Then $\left\langle z,X\right\rangle \geq
s(\log m)^{1/2}|\Omega |$. By Corollary \ref{contact pts} and the triangle
inequality, $\left\Vert z\right\Vert _{A^{\circ }}\leq c|\Omega |$ and $%
|z|=|\Omega |^{1/2}$. By the bound on $|\Omega |$, $c|\Omega |\geq t|\Omega
|^{1/2}$, which implies that $\left\Vert z\right\Vert _{A_{t}^{\circ }}\leq
c|\Omega |$. We thus have $\left\Vert X\right\Vert _{A_{t}}\geq \left\langle
z,X\right\rangle /\left\Vert z\right\Vert _{A_{t}^{\circ }}\geq cs(\log
m)^{1/2}$ (with prob. $\geq 0.52$). With probability at least $0.99$, $%
|X|\leq c\sqrt{n}$. Therefore, with probability at least $0.51$, $\left\Vert
\theta \right\Vert _{A_{t}}\geq cs\sqrt{(\log n)/n}$. Since this probability
bound is strictly greater than $0.5$, this leads to a bound on the median of 
$\left\Vert \cdot \right\Vert _{A_{t}}$ on $S^{n-1}$, i.e. $med\left(
\left\Vert \cdot \right\Vert _{A_{t}}\right) \geq cs\sqrt{(\log n)/n}$. Now
consider a linear transformation $T:\mathbb{R}^{n}\rightarrow \mathbb{R}^{n}$
such that $K=TA$ is in John's position (this is the final
position/coordinate structure as in the statement of the lemma). By
Corollary \ref{contact pts}, $c_{3}|x|\leq |Tx|\leq c_{4}|x|$ for all $x\in 
\mathbb{R}^{n}$. Recalling the definition $K_{t}=conv\left\{
tB_{2}^{n},K\right\} $, which implies that $\left\Vert x\right\Vert
_{K_{t}^{\circ }}=\max \left\{ ||x||_{K^{\circ }},t|x|\right\} $, we see
that $c_{3}\left\Vert x\right\Vert _{A_{t}^{\circ }}\leq \left\Vert
x\right\Vert _{K_{t}^{\circ }}\leq c_{4}\left\Vert x\right\Vert
_{A_{t}^{\circ }}$ for all $x\in \mathbb{R}^{n}$, and therefore $%
c_{3}\left\Vert x\right\Vert _{A_{t}}\leq \left\Vert x\right\Vert
_{K_{t}}\leq c_{4}\left\Vert x\right\Vert _{A_{t}}$. This then implies that $%
M_{t}\geq cs\sqrt{(\log m)/m}$, and clearly $b_{t}\leq t^{-1}$. Recalling
that $s=(1-2\log (ct)/\log m)^{1/2}$, these inequalities imply that%
\begin{equation*}
\frac{M_{t}}{b_{t}}\geq ct\left( 1-\frac{2\log (ct)}{\log m}\right)
^{1/2}\left( \frac{\log m}{m}\right) ^{1/2}=ct\left( \frac{\log m-2\log (ct)%
}{m}\right) ^{1/2}\geq ct\sqrt{\frac{1}{n}\log \left( \frac{c^{\prime }n}{%
t^{2}}\right) }
\end{equation*}
\end{proof}

\begin{proof}[Proof of Corollary \protect\ref{gen Dvoretzky type decomp}]
The idea of the proof is as follows: make repeated use of Lemma \ref%
{unconditional iso symm} (or just as well Corollary \ref{randomized
isomorphic Dvoretzky}) and Theorem \ref{general Dvoret} and then use the
union bound. The probability of a positive outcome for each group is at
least $1-Cn^{-c}$, and there are at most $C\log n$ groups. For the sake of
clarity, we choose to provide the details. Upper case letters such as $%
C,C^{\prime },C_{1}$ denote constants that are sufficiently large, while
lower case letters such as $c,c^{\prime },c_{1}$ etc. denote sufficiently
small constants in $(0,1)$. Without loss of generality we may take $n>n_{0}$
for some universal constant $n_{0}\in \mathbb{N}$. For each $A\subseteq
\{1,2,\ldots n\}$ let $x_{A}\in \mathbb{R}^{n}$ be such that $x_{i}=1$ if $%
i\in A$ and $x_{i}=0$ if $x\notin A$, and consider the corresponding
coordinate subspace $E_{A}=\left\{ x\in \mathbb{R}^{n}:i\notin A\Rightarrow
x_{i}=0\right\} $. Define%
\begin{equation*}
\left\Vert A\right\Vert _{0}:=\left\Vert x_{A}\right\Vert _{0}\hspace{0.65in}%
\left\Vert A\right\Vert _{cyc}:=\left\Vert x_{A}\right\Vert _{cyc}\hspace{%
0.65in}\left\Vert A\right\Vert _{Kol}:=\left\Vert x_{A}\right\Vert _{Kol}
\end{equation*}%
For each $1\leq j\leq \log _{2}\sqrt{c_{1}n}$, set%
\begin{eqnarray*}
\Lambda (1,j) &=&\left\{ A\subseteq \{1,2,\ldots n\}:\left\Vert A\right\Vert
_{0}^{1/2}=\left\lfloor c2^{j}\right\rfloor \right\} \\
\Lambda (2,j) &=&\left\{ A\subseteq \{1,2,\ldots n\}:\left( \frac{\left\Vert
A\right\Vert _{cyc}+\log n}{\log \left( 1+n\left\Vert A\right\Vert
_{cyc}^{-1}\right) }\right) ^{1/2}\leq c2^{j}\right\} \\
\Lambda (3,j) &=&\left\{ A\subseteq \{1,2,\ldots n\}:\left( \frac{\left\Vert
A\right\Vert _{0}+\left\Vert A\right\Vert _{Kol}+\log n}{\log \left(
1+n\left( \left\Vert A\right\Vert _{0}+\left\Vert A\right\Vert _{Kol}\right)
^{-1}\right) }\right) ^{1/2}\leq c2^{j}\right\}
\end{eqnarray*}%
Now let $U\in O(n)$ be a uniformly distributed random matrix. Consider the
events $\left\{ \Psi (i,j):1\leq i\leq 3,1\leq j\leq \log _{2}\sqrt{c_{1}n}%
\right\} $ where $\Psi (i,j)$ is the event that for all $A\in \Lambda (i,j)$
and all $a\in E_{A}$, $\frac{1}{2}M_{t}\left\vert a\right\vert \leq
\left\Vert Ua\right\Vert \leq \frac{3}{2}tM_{t}\left\vert a\right\vert $,
where $t=2^{j}$ and $M_{t}$ is the median of the Minkowski functional of $%
K_{t}=conv\{B_{X},tB_{2}^{n}\}$ restricted to $S^{n-1}$, as in Lemma \ref%
{unconditional iso symm}. Our task now is to bound the probability of these
events, and we do so separately for each value of $i$. \textbf{Case 1:}
\thinspace $i=1$ and $1\leq j\leq \log _{2}\sqrt{c_{1}n}$. We now use Lemma %
\ref{unconditional iso symm} with $t=2^{j}$ and Theorem \ref{general Dvoret}
with $\varepsilon =1/2$, $k=\left\lfloor c2^{j}\right\rfloor ^{2}<4c_{1}n$,%
\begin{equation*}
N=\left\vert \Lambda (1,j)\right\vert =\binom{n}{k}\leq \left( \frac{en}{k}%
\right) ^{k}
\end{equation*}%
and $(F_{i})_{1}^{N}$ any enumeration of $\left\{ E_{A}:A\in \Lambda
(1,j)\right\} $. This implies that with probability at least%
\begin{equation*}
1-C\exp \left( -ct^{2}\log \left( \frac{c^{\prime }n}{t^{2}}\right) \right)
\geq 1-Cn^{-c}
\end{equation*}%
the following event occurs: for all $A\in \Lambda (1,j)$ and all $a\in E_{A}$%
, $\frac{1}{2}M_{t}\left\vert a\right\vert \leq \left\Vert Ua\right\Vert
_{K_{t}}\leq \frac{3}{2}M_{t}\left\vert a\right\vert $ and therefore $\frac{1%
}{2}M_{t}\left\vert a\right\vert \leq \left\Vert Ua\right\Vert \leq \frac{3}{%
2}tM_{t}\left\vert a\right\vert $. i.e. $\mathbb{P}\left( \Psi (1,j)\right)
\geq 1-Cn^{-c}$. \textbf{Case 2:} \thinspace $i=2$ and $1\leq j\leq \log _{2}%
\sqrt{c_{1}n}$. We now follow the same procedure as in Case 1. Set $t=2^{j}$%
, and consider the function $\psi (s)=s^{2}\log \left( c^{\prime
}n/s^{2}\right) $. It follows from the inequality defining $\Lambda (2,j)$
and the fact that $\psi $ is increasing on $C\leq s\leq c\sqrt{n}$, that $%
\psi $ evaluated at%
\begin{equation*}
s=c^{-1}\left( \frac{\left\Vert A\right\Vert _{cyc}+\log n}{\log \left(
1+n\left\Vert A\right\Vert _{cyc}^{-1}\right) }\right) ^{1/2}
\end{equation*}%
is bounded above by $\psi \left( t\right) $, for any particular $A\in
\Lambda (2,j)$. Simplifying the resulting inequality yields $%
c^{-1}\left\Vert A\right\Vert _{cyc}\leq t^{2}\log \left( c^{\prime
}n/t^{2}\right) $. Here one also uses the fact that $\left\Vert A\right\Vert
_{cyc}<cn$, which follows directly from the definition of $\Lambda (2,j)$.
Then define $k=\left\lceil ct^{2}\log \left( c^{\prime }n/t^{2}\right)
\right\rceil $. Note that every $A\in \Lambda (2,j)$ is contained in a set
of the form%
\begin{equation*}
A_{m,k}^{\prime }=\left\{ (m+i)\enspace \text{mod} \enspace n:0\leq i\leq k-1\right\} 
\end{equation*}%
Thus we may take $N=n$ and apply Lemma \ref{unconditional iso symm} and
Theorem \ref{general Dvoret} with $\varepsilon =1/2$ and $%
F_{i}=A_{i,k}^{\prime }$. Our conclusion is that $\mathbb{P}\left( \Psi
(2,j)\right) \geq 1-Cn^{-c}$. \textbf{Case 3:} \thinspace $i=3$ and $1\leq
j\leq \log _{2}\sqrt{c_{1}n}$. This is similar to Case 2. Set $t=2^{j}$.
Repeating the above argument, we see that $\left\Vert A\right\Vert
_{0}+\left\Vert A\right\Vert _{Kol}\leq ct^{2}\log \left( c^{\prime
}n/t^{2}\right) $. We may then set $k=\left\lceil ct^{2}\log \left(
c^{\prime }n/t^{2}\right) \right\rceil $ and $\varepsilon =1/2$. Since any
binary string $b$ can be written as $b=\phi (p)$ for some other binary
string $p$ (here $\phi $ is the universal partial recursive function
involved in the definition of Kolmogorov complexity, see Section \ref%
{Background a7}), the number of strings $b$ with $C_{Kol}(b)\leq k$ is at
most $\left\vert \left\{ p\in \{0,1\}^{\ast }:\ell (p)\leq k\right\}
\right\vert \leq 2^{k+1}$. Therefore $N=\left\vert \Lambda (3,j)\right\vert
\leq 2^{k+1}$. It then follows after a third use of Lemma \ref{unconditional
iso symm} and Theorem \ref{general Dvoret} that $\mathbb{P}\left( \Psi
(3,j)\right) \geq 1-Cn^{-c}$. By the union bound%
\begin{equation*}
P\left( \Psi \right) \geq 1-Cn^{-c}\log (n)\geq 1-Cn^{-c}
\end{equation*}%
where $\Psi $ denotes the intersection of all $\Psi (i,j)$ as $1\leq i\leq 3$
and $1\leq j\leq \log _{2}\sqrt{c_{1}n}$. This completes the probabilistic
argument. For the remainder of the proof we assume that $\Psi $ occurs and
show that the conclusion of Theorem \ref{gen Dvoretzky type decomp} holds.
Consider any $a\in \mathbb{R}^{n}$, and the associated quantity $D(a)$ as
defined by (\ref{distotzion}). If the minimum in (\ref{distotzion}) is
attained at $\left\Vert a\right\Vert _{0}^{1/2}$ and $\left\Vert
a\right\Vert _{0}^{1/2}\leq c^{\prime }n$ then there exists $1\leq j\leq
\log _{2}\sqrt{c_{1}n}$ such that $\left\Vert a\right\Vert _{0}^{1/2}\leq
\left\lfloor c2^{j}\right\rfloor $ and $a\in E_{A}$ for some $A\in \Lambda
(1,j)$. Since $\Psi (1,j)\supseteq \Psi $ occurs, $\frac{1}{2}%
M_{t}\left\vert a\right\vert \leq \left\Vert Ua\right\Vert \leq \frac{3}{2}%
tM_{t}\left\vert a\right\vert $, with $t=2^{j}\leq C^{\prime }\left\Vert
a\right\Vert _{0}^{1/2}$ and the Theorem holds. Otherwise, if $\left\Vert
a\right\Vert _{0}^{1/2}>c^{\prime }n$ then by John's theorem $%
n^{-1/2}\left\vert a\right\vert \leq \left\Vert Ua\right\Vert \leq
\left\vert a\right\vert $, and the Theorem still holds. Similar arguments
hold when the minimum is attained at either of the other two terms involving 
$\left\Vert a\right\Vert _{cyc}$ and $\left\Vert a\right\Vert _{Kol}$.
\end{proof}

\begin{corollary}
\label{infinite decomp}Let $X$ be an infinite dimensional Banach space over $%
\mathbb{R}$ with a Schauder basis $(e_{i})_{1}^{\infty }$. For any $N\in 
\mathbb{N}$ and $\varepsilon >0$, $X$ admits an $FDD$ (finite dimensional
decomposition) $(E_{n})_{1}^{\infty }$ where $\dim (E_{n})\geq N$ and $%
d_{2}(E_{n})\leq (1+\varepsilon )$.
\end{corollary}

\begin{proof}
For all $n\in \mathbb{N}$ define $U_{n}=span\{e_{j}:(n-1)\exp (cN)<j\leq
n\exp (cN)\}$. Then apply Corollary \ref{gen Dvoretzky type decomp} followed
by Theorem \ref{general Dvoret} to $U_{n}$ to obtain $U_{n}=V_{1}^{(n)}%
\oplus V_{2}^{(n)}\ldots \oplus V_{N(n)}^{(n)}$, where $d_{BM}(V_{i}^{(n)},%
\ell _{2}^{k_{i}})\leq 1+\varepsilon $, with $k_{i}=\dim (V_{i}^{(n)})$. We
now claim that%
\begin{equation*}
X=V_{1}^{(1)}\oplus V_{2}^{(1)}\ldots \oplus V_{N(1)}^{(1)}\oplus
V_{1}^{(2)}\oplus V_{2}^{(2)}\ldots \oplus V_{N(2)}^{(2)}\oplus \ldots 
\end{equation*}%
The main subtlety here is convergence, however it follows using John's
theorem that the norms of the partial sum projections of $U_{n}$ onto $%
\oplus _{i=1}^{k}V_{i}^{(n)}$ for $1\leq k\leq N(n)$ are all bounded above
by $e^{cN}$.
\end{proof}

It was shown by Schechtman and Schmuckenschl\"{a}ger \cite{ScSc} that if $%
K\subset \mathbb{R}^{n}$ is a convex body in John's position, then for all $%
t\geq 0$%
\begin{equation}
\mathbb{P}\left\{ \left\Vert G\right\Vert _{K}\leq t\right\} \leq \mathbb{P}%
\left\{ \left\Vert G\right\Vert _{\infty }\leq t\right\}   \label{schsch}
\end{equation}%
where $G$ is a standard normal random vector in $\mathbb{R}^{n}$. Using
Lemma \ref{unconditional iso symm} we may recover a very similar type of
estimate.

\begin{corollary}
\label{concentration bound}There exist universal constants $%
C,c,c_{1},c_{2}>0 $ such that the following holds. Let $K\subset \mathbb{R}%
^{n}$ be a symmetric convex body in John's position. Then for all $1\leq
t\leq c^{\prime }\sqrt{\log n}$,%
\begin{equation*}
\sigma _{n}\left\{ \theta \in S^{n-1}:\left\Vert \theta \right\Vert _{K}\leq 
\frac{t}{\sqrt{n}}\right\} <C\exp \left( -c_{1}n\exp \left(
-c_{2}t^{2}\right) \right)
\end{equation*}%
where $\sigma _{n}$ is normalized Haar measure on $S^{n-1}$.
\end{corollary}

\begin{proof}
We first assume that $C^{\prime }\leq t\leq c\sqrt{\log n}$. Set $s=c\sqrt{n}%
\exp \left( -ct^{2}\right) $ and consider $K_{s}=conv\left\{
K,sB_{2}^{n}\right\} $. By Lemma \ref{unconditional iso symm} 
\begin{equation*}
M_{s}\geq c\sqrt{\frac{1}{n}\log \left( \frac{c^{\prime }n}{s^{2}}\right) }=%
\frac{2t}{\sqrt{n}}
\end{equation*}%
and $Lip\left( \left\Vert \cdot \right\Vert _{K_{s}}\right) \leq s^{-1}$. By
L\'{e}vy's inequality,%
\begin{equation*}
\sigma _{n}\left\{ \theta \in S^{n-1}:\frac{1}{2}M_{s}\leq \left\Vert \theta
\right\Vert _{K_{s}}\leq \frac{3}{2}M_{s}\right\} \geq 1-C\exp \left(
-cns^{2}M_{s}^{2}\right) \geq 1-C\exp \left( -cnt^{2}\exp \left(
-ct^{2}\right) \right)
\end{equation*}%
and the desired bound follows because $\left\Vert \theta \right\Vert
_{K}\geq \left\Vert \theta \right\Vert _{K_{s}}$. If $1\leq t<C^{\prime }$,
then the result follows by readjusting the constants in the bound.
\end{proof}

We refer to \cite{KlVe, LaOl} and Theorem 3.1 in \cite{FrGu} for related
small ball estimates.

For a convex body $K\subset \mathbb{R}^{n}$ with $0\in int(K)$, the
parameter $d_{u}(K)$ is defined for each $u>1$ as%
\begin{equation*}
d_{u}(K)=\min \left\{ n,-\log \sigma _{n}\left\{ \theta \in
S^{n-1}:\left\Vert \theta \right\Vert _{K}\leq \frac{1}{u}M\right\} \right\} 
\end{equation*}%
where $M$ is the mean of $\left\Vert \cdot \right\Vert _{K}$ on $S^{n-1}$.\
It was shown by Klartag and Vershynin \cite{KlVe} that the outer inclusion
in the randomized Dvoretzky theorem holds with uniformly bounded distortion
for all dimensions $1\leq l\leq c(u)d_{u}(K)$. It is known that $%
d_{u}(K)\geq ck(K)$, for $u\geq 2$ say, where $k(K)=nM^{2}$ (here $b=1$),
and that for $1\leq p\leq \infty $, $d_{C_{p}}(B_{p}^{n})\geq c_{p}n$ where $%
c_{p},C_{p}>0$ depend only on $p$ (and can be taken independent of $p$ for $%
1\leq p\leq 2$).

\begin{corollary}
\label{parameter est}For all $\varepsilon \in (0,1/2)$ and all $T>0$ there
exists $u_{0}\leq CT/\sqrt{\varepsilon }$ such that for all $n\geq \exp
\left( C\varepsilon ^{-1}\right) $ and any symmetric convex body $K\subset 
\mathbb{R}^{n}$ in John's position with%
\begin{equation*}
M(K)\leq T\sqrt{\frac{\log n}{n}}
\end{equation*}%
we have%
\begin{equation*}
d_{u_{0}}(K)\geq cn^{1-\varepsilon }
\end{equation*}
\end{corollary}

\begin{proof}
Set $t=c\sqrt{\varepsilon \log n}$ and $u_{0}=M\sqrt{n}t^{-1}\leq CT/\sqrt{%
\varepsilon }$. The result then follows from Corollary \ref{concentration
bound} or Equation (\ref{schsch}).
\end{proof}

\section{Almost-isometric theory}

\subsection{\label{ai dec}Almost-isometric decompositions in John's position}

For the entirety of this subsection, let $(X,\left\Vert \cdot \right\Vert )$
denote a real normed space of dimension $n\in \mathbb{N}$ that we identify
with $\mathbb{R}^{n}$ so that $B_{X}$ is in John's position. For any
subspace $E\subset X$ let $b(E)$ denote the Lipschitz constant of $%
\left\Vert \cdot \right\Vert $ restricted to $E$, and $M^{\sharp }(E)$ the
median of $\left\Vert \cdot \right\Vert $ restricted to $S^{n-1}\cap E$. Let 
$M^{\sharp }=M^{\sharp }(X)$. All random objects that we consider (points,
subspaces and orthogonal matrices) are distributed according to Haar measure
on the appropriate space. Lemmas \ref{alpha} and \ref{beta}, as well as the
proof of Theorem \ref{new} are taken (with permission) from \cite{Tikh},
with minor modifications.

\begin{lemma}
\label{alpha}There exists $c>0$ such that the following is true. Let $k\geq
(M^{\sharp })^{2}n$ and $F\in G_{n,k}$. Let $U\in O(n)$ be a random
orthogonal matrix and let $E=UF$. Then with probability at least $1-2^{-k}$, 
$b(E)\leq c\sqrt{k/n}$.
\end{lemma}

\begin{proof}
Let $\mathcal{N}$ be a $1/2$-net in $S^{n-1}\cap F$ with $\left\vert 
\mathcal{N}\right\vert \leq 6^{k}$. The result follows by applying Levy's
inequality (\ref{Levy's ineq}) with $t=c_{3}\sqrt{k/n}$ and applying the
series representation (\ref{net series}), the union bound, and the triangle
inequality.
\end{proof}

\begin{lemma}
\label{beta}There exists a universal constant $c_{1}>0$ such that the
following is true. Let $\left\Vert \cdot \right\Vert $ be a norm on $\mathbb{%
R}^{n}$ with the unit ball in John's position, let $\theta \in S^{n-1}$ be a
random point and $E\in G_{n,k}$ a random subspace (for some $k<n$). Let $%
\varepsilon >0$ be such that%
\begin{equation*}
\mathbb{P}\left\{ \left\vert \left\Vert \theta \right\Vert -M^{\sharp
}\right\vert \geq \varepsilon M^{\sharp }\right\} \leq 1/4
\end{equation*}%
Then%
\begin{equation*}
\mathbb{P}\left\{ \left\vert M^{\sharp }(E)-M^{\sharp }\right\vert \leq
\varepsilon M^{\sharp }\right\} \geq 1-2\exp (-c_{1}k)
\end{equation*}
\end{lemma}

\begin{proof}
Let $(v_{i})_{1}^{k}$ be i.i.d. random points on $S^{n-1}$. Let $\alpha =%
\mathbb{P}\left\{ M^{\sharp }(E)\leq (1-\varepsilon )M^{\sharp }\right\} $
and%
\begin{equation*}
\mathcal{M}=\left\{ H\in G_{n,k}:M^{\sharp }(H)\leq (1-\varepsilon
)M^{\sharp }\right\}
\end{equation*}%
By definition of $\mathcal{M}$, for all $1\leq i\leq k$ and any $H\in 
\mathcal{M}$ we have the 'conditional' probability $\mathbb{P}\left\{
\left\Vert v_{i}\right\Vert \leq (1-\varepsilon )M^{\sharp }:v_{i}\in
H\right\} \geq 1/2$, and therefore%
\begin{equation*}
\mathbb{P}\left\{ \left\vert \left\{ i:\left\Vert v_{i}\right\Vert \leq
(1-\varepsilon )M^{\sharp }\right\} \right\vert \geq k/2\right\} \geq \alpha
/2
\end{equation*}%
Of course the event $\left\{ v_{i}\in H\right\} $ has measure zero and this
does not fit into the classical definition of conditional probability.
However, it can be justified using a construction involving Fubini's theorem
on $O(n)\times O(k)$. On the other hand, from the condition imposed on $%
\varepsilon $, $\mathbb{P}\left\{ \left\Vert v_{i}\right\Vert \leq
(1-\varepsilon )M^{\sharp }\right\} \leq 1/4$ and%
\begin{equation*}
\mathbb{E}\left\vert \left\{ i:\left\Vert v_{i}\right\Vert \leq
(1-\varepsilon )M^{\sharp }\right\} \right\vert \leq k/4
\end{equation*}%
By Hoeffding's inequality,%
\begin{equation*}
\mathbb{P}\left\{ \left\vert \left\{ i:\left\Vert v_{i}\right\Vert \leq
(1-\varepsilon )M^{\sharp }\right\} \right\vert \geq k/2\right\} \leq \exp
(-c_{1}k)
\end{equation*}%
The bound on $\mathbb{P}\left\{ M^{\sharp }(E)\geq (1+\varepsilon )M^{\sharp
}\right\} $ follows similar lines.
\end{proof}

\begin{proof}[Proof of Theorem \protect\ref{new}]
If $M^{\sharp }\geq c_{1}\varepsilon ^{-1}(\log (n)/n)^{1/2}$, the statement
follows from Theorem \ref{general Dvoret} and we may assume without loss of
generality that $M^{\sharp }<c_{1}\varepsilon ^{-1}(\log (n)/n)^{1/2}$. Let%
\begin{equation*}
N=\left\lfloor n^{1-\varepsilon ^{2}(\log \varepsilon
^{-1})^{-2}}\right\rfloor
\end{equation*}%
and let $H_{1}\oplus H_{2}\ldots \oplus H_{N}$ be a decomposition of $%
\mathbb{R}^{n}$ into mutually orthogonal subspaces of dimension either $k$
or $k+1$, with $k\approx n^{\varepsilon ^{2}(\log \varepsilon ^{-1})^{-2}}$
and let $U\in O(n)$ be a random orthogonal matrix. From Lemmas \ref{alpha}
and \ref{beta}, it follows that with high probability, for all $i$, $%
M^{\sharp }(UH_{i})/b(UH_{i})\geq c_{4}\varepsilon ^{-1}(\log \varepsilon
^{-1})\sqrt{(\log k)/k}$. It now follows by Theorem \ref{general Dvoret}
that each $UH_{i}$ can be decomposed yet again into approximately Euclidean
subspaces of dimension $c\varepsilon ^{2}(\log \varepsilon ^{-1})^{-1}\log n$%
.
\end{proof}

\subsection{Grid structures}

\begin{lemma}
\label{c0 not loc Hilbert}The space $c_{0}$ does not satisfy Definition \ref%
{def loc hilb}.
\end{lemma}

\begin{proof}
Assume for the sake of a contradiction that $c_{0}$ satisfies Definition \ref%
{def loc hilb}. Let $n\geq 3$, $k=2$ and consider any $0<\varepsilon <(\sqrt{%
2}-1)^{4}$. This ensures that both of the following inequalities hold%
\begin{eqnarray}
(1+\varepsilon +2\sqrt{\varepsilon }) &<&\sqrt{2}(1-\varepsilon )
\label{little a} \\
(1+\varepsilon )^{2} &<&2(1-\varepsilon )^{2}  \label{little b}
\end{eqnarray}%
Let $(f_{i})_{1}^{n}$ be a basis for $\ell _{\infty }^{n}$ such that for any 
$a\in \mathbb{R}^{n}$ with $\left\Vert a\right\Vert _{0}\leq 2$, (\ref{loc
hilbert def}) holds. Let $T$ be the $n\times n$ matrix with the vectors $%
(f_{i})_{1}^{n}$ as columns. For all $2$-sparse vectors $x\in \mathbb{R}^{n}$%
, (\ref{loc hilbert def}) can be written as 
\begin{equation}
(1-\varepsilon )|x|\leq ||Tx||_{\infty }\leq (1+\varepsilon )|x|
\label{sparse ineq a}
\end{equation}%
It follows by setting $x=e_{j}$ (the $j^{th}$ standard basis vector of $%
\mathbb{R}^{n}$) that for all $1\leq j\leq n$%
\begin{equation*}
1-\varepsilon \leq \max_{1\leq i\leq n}|T_{i,j}|\leq 1+\varepsilon
\end{equation*}%
In particular, there exists $\nu =\nu (j)$ such that $|T_{\nu ,j}|\geq
1-\varepsilon $. Since this holds for each column, there are at least $n$
entries of the matrix such that $|T_{i,j}|\geq 1-\varepsilon $. By duality,
for all $1\leq i\leq n$ and any $j,l\in \{1,2\ldots n\}$, 
\begin{equation*}
T_{i,j}^{2}+T_{i,l}^{2}\leq (1+\varepsilon )^{2}
\end{equation*}%
By (\ref{little b}), there can be at most one such entry per row, and we
conclude that there is exactly one in every row. Hence $\nu \in S_{n}$ is a
permutation of the set $\{1,2\ldots n\}$. For all $1\leq i\leq n$, if $j\neq
\nu ^{-1}(i)$ then $T_{i,j}^{2}=T_{i,j}^{2}+T_{i,\nu ^{-1}(i)}^{2}$ $%
-T_{i,\nu ^{-1}(i)}^{2}\leq (1+\varepsilon )^{2}-(1-\varepsilon
)^{2}=4\varepsilon $ and $|T_{i,j}|\leq 2\sqrt{\varepsilon }$. The matrix $T$
is therefore a small perturbation of a permutation matrix. If we now take $%
x=e_{1}+e_{2}$, then $|x|=\sqrt{2}$ and for all $1\leq i\leq n$, 
\begin{equation*}
\left\vert T_{i,1}+T_{i,2}\right\vert \leq 1+\varepsilon +2\sqrt{\varepsilon 
}<\sqrt{2}(1-\varepsilon )
\end{equation*}%
by (\ref{little a}), which implies that $||Tx||_{\infty }<\sqrt{2}%
(1-\varepsilon )$. This contradicts (\ref{sparse ineq a}) and the result
follows.
\end{proof}

\begin{proof}[Proof of Theorem \protect\ref{local Hilbert cotype}]
If $X$ fails to have nontrivial cotype, then it follows from the
Maurey-Pisier-Krivine theorem that $\ell _{\infty }$ (equivalently $c_{0}$)
is finitely representable in $X$, and therefore $X$ fails to satisfy
Definition \ref{def loc hilb} by Lemma \ref{c0 not loc Hilbert}. On the
other hand if $X$ has cotype $q$ for some $q<\infty $, and cotype constant $%
\beta \in (0,1]$, then it follows by Theorem \ref{general Dvoret} that $X$
satisfies Definition \ref{def loc hilb}. Here we also use the fact that $%
\binom{n}{k}\leq (en/k)^{k}$, and the bound $M\geq c\beta n^{1/q-1/2}$ from
equation (\ref{cotype bound}). The corresponding probability bound is
positive if $k\leq c\beta ^{2}\varepsilon ^{2}\left( \log (en/k)\right)
^{-1}n^{2/q}$. In the case $q=2$, this bound can be polished. A simple
calculation shows that if $0<s<1/3$ and $T\geq \max \left\{ e,3s^{-1}\log
s^{-1}\right\} $, then $T^{-1}\log T\leq s$. Applying this with $T=n/k$ we
see that $k\leq c\beta ^{2}\varepsilon ^{2}\left( \log \beta ^{-1}+\log
\varepsilon ^{-1}\right) ^{-1}n$ is sufficient.
\end{proof}

\section{Remaining proofs}

We state the following result for completeness and because we could not find
exactly what we wanted in the literature. See \cite{BDDW} for a similar
statement that applies to a wider class of random matrices, but has slightly
weaker dependence on $\varepsilon $.

\begin{lemma}
\label{standard RIP}Consider any $m,n,k\in \mathbb{N}$ and $0<\varepsilon
<0.99$ such that $k\leq c\varepsilon ^{2}(\log (en/k))^{-1}m$. Let $U$ be a
random $m\times n$ matrix with i.i.d. $N(0,1)$ entries. With probability at
least $1-c_{1}\exp (-c_{2}m\varepsilon ^{2})$, the inequality $%
(1-\varepsilon )|x|\leq m^{-1/2}|Ux|\leq (1+\varepsilon )|x|$ holds
simultaneously for all $k$-sparse vectors $x\in \mathbb{R}^{n}$.
\end{lemma}

\begin{proof}
Let $E\in G_{n,k}$ be any $k$-dimensional coordinate subspace of $\mathbb{R}%
^{n}$. By Schechtman's version of the general Dvoretzky theorem (Theorem 2
in \cite{Sch0}, or Theorem \ref{general Dvoret} here for the probabilistic
statement), the required bound holds for all $x\in E$ with probability at
least $1-c_{1}\exp \left( -c_{2}m\varepsilon ^{2}\right) $. By Stirling's
approximation there are at most $\binom{n}{k}\leq (en/k)^{k}$ such
coordinate subspaces, and the result follows from the union bound.
\end{proof}

\begin{proof}[Proof of Proposition \protect\ref{RIP for cotype}]
We know from Theorem \ref{local Hilbert cotype} that there exists a basis in 
$X$, say $(e_{i})_{1}^{n}$ such that for all $k$-sparse vectors $a\in 
\mathbb{R}^{n}$,%
\begin{equation*}
(1-\varepsilon /6)|a|\leq \left\Vert \sum_{i=1}^{n}a_{i}e_{i}\right\Vert
_{X}\leq (1+\varepsilon /6)|a|
\end{equation*}%
Identify $X$ with $\mathbb{R}^{n}$ using this basis. Identify $Y$ with $%
\mathbb{R}^{m}$ in such a way so that $B_{Y}$ is in John's position, and
then readjust the coordinate structure by scalar multiplication so that the
mean of $\left\Vert \cdot \right\Vert _{Y}$ in $S^{m-1}$ obeys $M(\left\Vert
\cdot \right\Vert _{Y})=1$. Let $G$ be a random $m\times n$ matrix with
i.i.d. $N(0,1)$ entries. For each $k$-dimensional coordinate subspace $%
E\subset X$, $GE$ has dimension $k$ with probability 1 and is uniformly
distributed in $G_{m,k}$. Therefore, using Theorem \ref{general Dvoret} and (%
\ref{cotype bound}), with probability at least $1-c_{1}\exp \left( -c\beta
_{2}^{2}m^{2/q_{2}}\varepsilon ^{2}\right) $, for all $y\in GE$, $%
(1-\varepsilon /6)|y|\leq \left\Vert y\right\Vert _{Y}\leq (1+\varepsilon
/6)|y|$. By applying the union bound, with probability at least $%
1-c_{1}n^{k}\exp \left( -c\beta _{2}^{2}m^{2/q_{2}}\varepsilon ^{2}\right) $%
, for all $k$-sparse vectors $x\in X$ we have $(1-\varepsilon /6)|Gx|\leq
\left\Vert Gx\right\Vert _{Y}\leq (1+\varepsilon /6)|Gx|$. The result now
follows from Lemma \ref{standard RIP} and the inequality $1-\varepsilon \leq
(1-\varepsilon /6)(1+\varepsilon /6)^{-2}\leq (1-\varepsilon
/6)^{-2}(1+\varepsilon /6)\leq 1+\varepsilon $.
\end{proof}

\begin{proof}[Proof of Proposition \protect\ref{general JL}]
The difference between any two $k$-sparse vectors is $2k$-sparse. The result
now follows from Theorem \ref{local Hilbert cotype} and the comment
following the statement of the theorem (by readjusting the constants
involved) as well as the usual form of the Johnson-Lindenstrauss lemma.
\end{proof}

\section*{Acknowledgements}

The author would like to thank Olivier Gu\'{e}don, Bo'az Klartag, Alexander
Koldobsky, Alexander Litvak, Grigoris Paouris, Anup Rao, Mark Rudelson,
Stanis\l aw Szarek, Konstantin Tikhomirov, Nicole Tomczak-Jaegermann, Petros
Valettas, Roman Vershynin, and the anonymous referees for comments related
to the paper. We also thank Beatrice-Helen Vritsiou for pointing out a flaw
in the original manuscript.

\end{document}